\documentclass[11pt]{amsart}

\usepackage{amsmath}
\usepackage{amssymb}
\usepackage{graphicx}
\usepackage{footmisc}

\newtheorem{them}{Theorem}[section]

\newtheorem{coro}{Corollary}[section]
\newtheorem{pro}{Proposition}[section]

\numberwithin{equation}{section}

\begin{document}

\title[Second-order Optimality Conditions and Applications]{Second-order Optimality Conditions by Generalized Derivatives and Applications in Hilbert Spaces} %

\author[Z. Wei]{Zhou Wei}
\address[Z. Wei]{Department of Mathematics, Yunnan University, Kunming 650091, People's Republic of China}
\email{\tt  wzhou@ynu.edu.cn}

\author[J.-C. Yao]{Jen-Chih Yao$^*$}

\address[J.-C. Yao]{Center for General Education, China Medical University, Taichung 40402, Taiwan}
\email{{\tt yaojc@mail.cmu.edu.tw}}


\renewcommand{\thefootnote}{Corresponding author.}

\begin{abstract}
In this paper, in terms of three types of generalized second-order derivatives of a nonsmooth function, we mainly study the corresponding second-order optimality conditions in a Hilbert space and prove the equivalence among these optimality conditions for paraconcave functions. As applications, we use these second-order optimality conditions to study strict local minimizers of order two and provide sufficient and/or necessary conditions for ensuring the local minimizer. This work extends and generalizes the study on second-order optimality conditions from the finite-dimensional space to the Hilbert space.
\end{abstract}

\maketitle

\let\thefootnote\relax\footnotetext{$^*$Corresponding author.}
\let\thefootnote\relax\footnotetext{This work of the first author was supported by the National Natural Science Foundation of P. R. China (grant 11401518), the Fok Ying-Tung Education Foundation (grant 151101) and the Scientific Research Foundation from Education Department of Yunnan Province under grant 2015Z009.}

\section{Introduction}

Variational analysis has been recognized as a broad spectrum of mathematical theory that has grown in connection with the study of problems on optimization, equilibrium, control and stability of linear and nonlinear systems, and its focus is mainly on optimization of functions relative to various constraints and on sensitivity or stability of optimization-related problems with respect to perturbation. Since nonsmooth optimization problems by nonsmooth functions, sets with nonsmooth boundaries or set-valued mappings frequently appear in variational theory and its application, nonsmooth analysis in variational analysis has played an important role in such aspects of mathematical programming and optimization (cf. \cite{2,7,8,17,18} and references therein). Over the past several decades, the first-order nonsmooth analysis has been extensively and systemically studied by many authors in both finite-dimensional and infinite-dimensional spaces, and also fruitfully applied to many aspects of applied mathematics such as first-order optimality conditions, sensitivity analysis, constrained optimization, equilibrium problems with nonsmooth data and optimal control (cf. \cite{AF,2,7,8,16,19}). However, the literature in dealing with second-order nonsmooth analysis is not too much relative to the first-order analysis. We refer readers to books \cite{2,16,19} for the application of second-order generalized differential constructions to optimization, sensitivity and related problems. Given a nonsmooth function defined on a Hilbert space, we mainly study three types of generalized second-order derivatives in this paper: {\it second-order lower Dini-directional derivative}, {\it second-order mixed graphical derivative} and {\it second-order mixed proximal subdifferential} (see Section 3). Then we use these second-order derivatives to consider second-order optimality conditions and investigate their equivalent interrelationship.

Second-order optimality conditions have played important roles in mathematical programming and have been extensively studied by many authors (cf. \cite{4, 5,6,EM,13,PR,24,27} and references therein). Recently Eberhard and Wenczel \cite{11} and Eberhard and Mordukhovich \cite{EM} discussed three different types of second-order optimality conditions which are based on generalized second-order directional derivative, graphical derivative of proximal subdifferential and second-order proximal subdifferential defined via coderivative of proximal subdifferential. The equivalence among these optimality conditions for paraconcave functions is also proved. Using these three types of second-order derivatives aforementioned, we are inspired by \cite{EM,11} to continue studying second-order optimality conditions (with some minor modifications) in a Hilbert space, and mainly study the interrelationship among them. It is also proved that the equivalence among these optimality conditions for paraconcave functions is still valid in the Hilbert space setting. As applications, we use these second-order optimality conditions to investigate strict local minimizers of order two for extended real-valued nonsmooth functions in the Hilbert space.

The notion of strict minimizer of order two for a nonsmooth function has been proved to be useful in optimization and relates closely with the convergence of numerical procedures. Hestenes \cite{12} considered this notion and used it to prove sufficient optimality conditions. Cromme \cite{10} and Auslender \cite{1} studied this notion in connection with convergence of numerical procedures and provided stability conditions. Studniarski \cite{22} used first and second order lower Dini-directional derivatives to study the local strict minimizer of order two and established necessary and sufficient second-order optimality conditions. Along the line given in \cite{22}, Ward \cite{23} investigated another derivatives and tangent cones to study local strict minimizer of order two and optimality conditions. This notion has also been generalized in the senses of weak sharp minima and $\phi$-minima and was extended to vector optimization problems and set-valued mappings (see \cite{9, 14} and references therein).

Note that Eberhard and Wenczel \cite{11} discussed strict local minimizer of order two in a finite-dimension space and provided its characterizations in terms of second-order optimality conditions. Along this line, as one main goal of this paper, we apply the second-order optimality conditions to the strict local minimizer of order two and aim to establish its characterizations in the Hilbert space. These characterizations reduce to the existing ones when restricted to the finite-dimensional space.

The paper is organized as follows. In Section 2, we will give some definitions and preliminaries used in this paper. Our notation is basically standard and conventional in the area of variational analysis. Section 3 is devoted to three types of second-order derivatives of a nonsmooth function and their important properties. In Section 4, by using these second-order derivatives, we mainly study several kinds of second-order optimality conditions in a Hilbert space and present results on their equivalence interrelationship. In Section 5, we first present a counterexample to show that the existing theorem on strict local minimizers of order two given in finite-dimensional space is not valid for the Hilbert space case (see Example 5.1), and then apply these second-order optimality conditions to characterizing the strict local minimizers for this case. The conclusion of this paper is presented in Section 6.

\section{Preliminaries}
Let $H$ be a {\em Hilbert space} equipped with the {\em inner product}
$\langle\cdot, \cdot\rangle$ and the corresponding {\em norm}
$\|\cdot\|$, respectively. Denote by $B_H$ and $S_H$ the {\em unit closed ball} and the {\em unit sphere} of $H$, respectively. For $x\in
H$ and $\delta>0$, let $B(x, \delta)$ denote the {\em open ball}
with center $x$ and radius $\delta$.

Given a multifunction $F: H \rightrightarrows H$, the symbol
\begin{equation*}
\begin{array}r
\mathop{\rm Limsup}\limits_{y\rightarrow x}F(y):=\Big\{\zeta\in H: \exists \ {\rm sequences} \ x_n\rightarrow x \ {\rm and} \ \zeta_n\stackrel{w}\longrightarrow \zeta \ {\rm with}\  \\
\zeta_n\in F(x_n) \ {\rm for\ all \ } n\in \mathbb{N} \Big\}
\end{array}
\end{equation*}
signifies the {\it sequential Painlev\'{e}-Kuratowski outer/upper limit} of $F(x)$ as $y\rightarrow x$, where $
\zeta_n\stackrel{w}\longrightarrow \zeta$ means $\{\zeta_n\}$ converges weakly to $\zeta$.

Given a set $A\subset H$, we denote by $\delta_A(\cdot)$ the {\it indicator function} of $A$ which is defined as $\delta_A(x)=0$ if $x\in A$ and $\delta_A(x)=+\infty$ if $x\not\in A$. We denote by $\overline A$ and $\overline A^w$ the {\em norm closure}
and the {\em weak closure} of $A$ in the norm topology and the weak topology
respectively, and denote by ${\rm aff}A$ the {\em affine hull} of
$A$. Let ${\rm ri}A$, ${\rm qri}A$ and ${\rm sqri}A$ denote the {\em relative
interior}, the {\em quasi-relative interior} and the {\em strong quasi-relative interior}, respectively, which are
defined by
\begin{eqnarray*}
{\rm ri}A:&=&\big\{x\in A: \exists\ \delta>0\ {\rm such \ that}\ B(x, \delta)\cap {\rm aff}A\subset A\big\},\\
{\rm qri}A:&=&\big\{x\in A: \overline{\rm cone}(A-x)\,\,{\rm is\,\,a\,\,subspace}\big\},\\
{\rm sqri}A:&=&\big\{x\in A: {\rm cone}(A-x)\,\,{\rm is\,\,a\,\,closed\,\,subspace}\big\}
\end{eqnarray*}
where cone$(A-x)$ denotes the cone generated by $A-x$ and $\overline{\rm cone}(A-x)$ denotes the closure of cone$(A-x)$. Clearly ${\rm sqri}A\subset {\rm qri}A$. When $H$ is
finite-dimensional, these three types of relative interior coincide; that is ${\rm sqri}A={\rm qri}A={\rm ri}A$. Readers are invited to refer to \cite{3} for more details on these interior concepts.

Let $S$ be a nonempty closed subset of $H$. Recall from \cite{AF} that the {\it contingent cone} and the {\it weak contingent cone} of $S$ at $x\in S$, denoted by $T(S, x)$ and $T^w(S, x)$ respectively, are defined by
\begin{eqnarray*}
T(S, x):=\big\{h\in H: \exists\ t_n\rightarrow 0^+\ {\rm and} \  h_n\rightarrow h\ {\it s.t.} \ x+t_nh_n\in S \ \forall n\in \mathbb{N}\big\},\, \ \\
T^w(S, x):=\big\{h\in H: \exists\ t_n\rightarrow 0^+\ {\rm and}  \ h_n\stackrel{w}\longrightarrow h\ {\it s.t.}\ x+t_nh_n\in S\ \forall n\in \mathbb{N}\big\}.
\end{eqnarray*}
When $H$ is a finite-dimensional space, both contingent cone and weak contingent cone coincide.

For any point $z\in H$, the distance between $z$ and $S$ is given by
$$
d(z, S):=\inf\{\|z-s\|\,:\,s\in S\}.
$$
Let $x\in S$. Recall from \cite{8} that the {\it proximal normal cone} of $S$ at $x$, denoted by $N^p(S, x)$, is defined as
\begin{equation}
N^p(S,x):=\{\zeta\in H:  \exists \ t>0\ {\rm such\ that} \ d(x+t\zeta, S)=t\|\zeta\|\}.
\end{equation}
It is known and easy to verify that $\zeta\in N^p(S,x)$ if and only if there exists $\sigma\in (0, +\infty)$ such that
\begin{equation}
\langle\zeta, s-x\rangle\leq \sigma\|s-x\|^2\ \ {\rm for\ all} \ s\in S.
\end{equation}

Let $\hat N(S, x)$ denote the {\it Fr\'{e}chet normal cone} of $S$ at $x$; that is,
$$
\hat N(S, x):=\left\{\zeta\in H: \limsup_{y\stackrel{S}\longrightarrow x}\frac{\langle \zeta, y-x\rangle}{\|y-x\|}\leq 0\right\}
$$
where $y\stackrel{S}\longrightarrow x$ means $y\rightarrow x$ and $y\in S$. Since a Hilbert space is reflexive, it is easy to verify that
\begin{equation}\label{2.3}
\hat N(S, x)=\big(T^w(S, x)\big)^{\circ}
\end{equation}
where $\big(T^w(S, x)\big)^{\circ}$ is the {\it dual cone} of $T^w(S, x)$ which is defined by
$$
\big(T^w(S, x)\big)^{\circ}:=\{v\in H: \langle v, h\rangle\leq 0\ {\rm for\ all} \ h\in T^w(S, x)\}.
$$
The {\it Mordukhovich(limiting/basic) normal cone} of $S$ at $x$, denoted by $N(S, x)$, is defined as
\begin{equation}
N(S, x):=\mathop{\rm Limsup}_{y\stackrel{S}\longrightarrow x}N^p(S, y).
\end{equation}
Thus, $\zeta\in N(S, x)$ if and only if there exists a sequence $\{(x_n, \zeta_n)\}$ in $S\times H$ such that $x_n\rightarrow x$, $\zeta_n\stackrel{w}\longrightarrow \zeta$ and $\zeta_n\in N^p(S, x_n)$ for each $n\in \mathbb{N}$.

Let $f:H\rightarrow
\mathbb{R}\cup\{+\infty\}$ be an extended real-valued and lower semicontinuous function. We denote
$$
{\rm dom}f:=\{y\in H: f(y)<+\infty\}\,\,\,{\rm and}\,\,\,{\rm epi}f:=\{(x, \alpha)\in H\times \mathbb{R}: f(x)\leq \alpha\}
$$
the {\it domain} and the {\it epigraph} of $f$, respectively. Let $x\in{\rm dom}f$. Recall that the {\it proximal subdifferential} of $f$ at $x$, denoted by $\partial_pf(x)$, is defined by
\begin{equation}
\partial_pf(x):=\{\zeta\in H: (\zeta, -1)\in N^p({\rm epi}f, (x, f(x)))\}.
\end{equation}
It is known from \cite{8} that $\zeta\in \partial_pf(x)$ if and only if there exist $\sigma, \delta\in (0, +\infty)$ such that
\begin{equation}
f(y)\geq f(x)+\langle\zeta,y-x\rangle-\frac{\sigma}{2}\|y-x\|^2\ \ {\rm for\ all} \
y\in B(x,\delta).
\end{equation}
The {\it Mordukhovich(limiting/basic) subdifferential} of $f$ at $x$ is defined as
\begin{equation}
\partial f(x):=\{\zeta\in H: (\zeta, -1)\in N({\rm epi}f, (x, f(x)))
\}.
\end{equation}
It is proved in \cite{16, 17} that
$$
\partial f(x)=\mathop{\rm Limsup}_{y\stackrel{f}\rightarrow x}\partial_pf(y)
$$
where $y\stackrel{f}\longrightarrow x$
means $y\rightarrow x$ and $f(y)\rightarrow f(x)$. Therefore $\zeta\in\partial f(x)$ if and only if there exist $x_n\stackrel{f}\longrightarrow x$ and $\zeta_n\stackrel{w}\longrightarrow\zeta$ such that $\zeta_n\in\partial_pf(x_n)$ for each $n\in\mathbb{\mathbb{N}}$.

When $f$ is convex, the proximal subdifferential and the limiting subdifferential of $f$ at $x\in{\rm dom}f$ coincide and both reduce to the  subdifferential in the sense of convex analysis, that is
$$
\partial_p f(x)=\partial f(x)=\{\zeta\in H: \langle \zeta, y-x\rangle\leq f(y)-f(x)\ \ {\rm for\ all} \ y\in H\}.
$$
Readers are invited to consult \cite{2,3,7,8,16,19} for more details on these various normal cones and subdifferentials.\\

The following concepts of {\it paraconcavity and paraconvexity} are used in our analysis.

For an extended-real-valued function $\varphi:H\rightarrow \mathbb{R}\cup\{+\infty\}$,  recall from \cite{11} that $\varphi$ is said to be {\it paraconvex}, if there exists $\lambda\in (0, +\infty)$ such that $\varphi+\frac{1}{2\lambda}\|\cdot\|^2$ is
convex on $H$ and $\varphi$ is said to be {\it locally paraconvex} around $\bar
x\in{\rm dom}(\varphi)$, if there exist $\delta, \lambda\in (0, +\infty)$ such that $\varphi+\frac{1}{2\lambda}\|\cdot\|^2$ is
convex relative to $B(\bar x, \delta)$. The function $\varphi$ is said to
be {\it locally paraconcave}, if $-\varphi$ is locally paraconvex.

\setcounter{equation}{0}

\section{Second-order derivatives of an extended real-valued function}
In this section, we consider several types of second-order derivatives of a nonsmooth function; namely, second-order lower Dini-directional derivative of a function, second-order mixed graphical derivative and second-order mixed proximal subdifferential of a function, and then study some properties of these second-order derivatives which will be used in our analysis.

Let $f:H\rightarrow \mathbb{R}\cup\{+\infty\}$ be a proper lower semicontinuous function. We denote by
$$
{\rm gph}(\partial_pf):=\{(x, u)\in H\times H: u\in \partial_pf(x)\}
$$
the {\it graph} of proximal subdifferential $\partial_pf$. In this paper, taking into account the application to second-order optimality conditions in the Hilbert space, we first study the following {\it mixed contingent cone} of ${\rm
gph}(\partial_pf)$ and its associated polar.\\

{\it Let $(\bar x, p)\in
{\rm gph}(\partial_pf)$. The {\it mixed contingent cone} of ${\rm
gph}(\partial_pf)$ at $(\bar x, p)$, denoted by $T_M({\rm
gph}(\partial_pf), (\bar x, p))$, is defined as follows:}
\begin{equation}\label{3.1}
\begin{array}r
(h, z)\in T_M({\rm gph}(\partial_pf), (\bar x, p)) \Leftrightarrow
\exists\, t_n\rightarrow 0^+,\,\,h_n\rightarrow h\,\,{\it and}\,\,
z_n\stackrel{w}\longrightarrow z \ {\it such\,\,that}\\
(\bar x+t_nh_n, p+t_nz_n)\in {\rm gph}(\partial_pf)\ \ {\it for\ all} \ n\in \mathbb{N}.
\end{array}
\end{equation}

By the definition, the following inclusions are trivial:
\begin{equation}\label{3.2}
T({\rm gph}(\partial_pf), (\bar x, p))\subset T_M({\rm gph}(\partial_pf), (\bar x, p))\subset T^w({\rm gph}(\partial_pf), (\bar x, p)).
\end{equation}

We denote by $D^2f(\bar x, p)(h)$ and $D_M^2f(\bar x, p)(h)$ the {\it second-order graphical derivative} and the {\it second-order mixed graphical derivative} of $f$ at $(\bar x, p)$ in the direction
$h\in H$, respectively which are defined as
\begin{eqnarray*}
D^2f(\bar x, p)(h):&=&\big\{z\in H: (h, z)\in T({\rm gph}(\partial_pf), (\bar x, p))\big\}\\
D_M^2f(\bar x, p)(h):&=&\big\{z\in H: (h, z)\in T_M({\rm gph}(\partial_pf), (\bar x, p))\big\}.
\end{eqnarray*}

We denote by $\hat{\partial}^2f(\bar x, p)(h)$ and $\partial_M^2f(\bar x, p)(h)$ the {\it second-order proximal subdifferential} and the {\it second-order mixed proximal subdifferential} of $f$ at $(\bar x, p)$ in the direction $h\in H$, respectively and they are defined by
\begin{eqnarray*}\label{4.10}
\hat{\partial}^2f(\bar x, p)(h):&=&\big\{z\in H: (z, -h)\in \hat{N}({\rm gph}(\partial_pf), (\bar x, p))\big\}\\
\partial_M^2f(\bar x, p)(h):&=&\big\{z\in H: (z, -h)\in N_M({\rm gph}(\partial_pf), (\bar x, p))\big\}
\end{eqnarray*}
where $N_M({\rm gph}(\partial_pf), (\bar x, p))$ is the dual cone of $T_M({\rm gph}(\partial_pf), (\bar x, p))$; that is
\begin{equation}\label{3.3}
N_M({\rm gph}(\partial_pf), (\bar x, p)):=\big(T_M({\rm gph}(\partial_pf), (\bar x, p))\big)^{\circ}.
\end{equation}
By \eqref{2.3}, \eqref{3.2} and \eqref{3.3}, one can easily verify that
\begin{equation}\label{3.4}
\hat N({\rm gph}(\partial_pf), (\bar x, p))\subset N_M({\rm gph}(\partial_pf), (\bar x, p)).
\end{equation}
When $H$ is finite-dimensional, for any $h\in H$, one has
\begin{equation}\label{3.5a}
D^2f(\bar x, p)(h)=D_M^2f(\bar x, p)(h)\ \ {\rm and}\ \ \hat{\partial}^2f(\bar x, p)(h)=\partial_M^2f(\bar x, p)(h)
\end{equation}
since $T_M({\rm gph}(\partial_pf), (\bar x, p))$ coincides with $T({\rm gph}(\partial_pf), (\bar x, p))$ in this case.

Recall that the {\it second-order lower Dini-directional derivative} of $f$ at $\bar x$
for $p\in\partial_pf(\bar x)$ along the direction $h\in H$ is defined as
\begin{equation}\label{3.5}
f_-''(\bar x, p, h):=\liminf_{h'\rightarrow h, t\downarrow 0}\Delta_2f(\bar x, p, t,
h')
\end{equation}
where
$$
\Delta_2f(\bar x, p, t,
u):=\frac{f(\bar x+tu)-f(\bar x)-t\langle p, u\rangle}{\frac{1}{2}t^2},\ \ \forall (t, u)\in (0, +\infty)\times H.
$$
Applying \cite[Theorem 19]{11}, for all $t>0$, one has
\begin{equation}\label{3.6}
\partial_p\big(\frac{1}{2}\Delta_2f(\bar x, p, t, \cdot)\big)(w)=\frac{1}{t}(\partial_pf(\bar x+tw)-p).
\end{equation}

The following proposition provides some properties on the
second-order lower Dini-directional derivative $f_-''(\bar x, p, \cdot)$.

\begin{pro}
Let $f:H\rightarrow \mathbb{R}\cup\{+\infty\}$ be a proper lower
semicontinuous function and $\bar x\in{\rm dom}f$ with $p\in\partial_pf(\bar x)$. Then

{\rm (i)} $f_-''(\bar
x, p, \cdot)$ is lower semicontinuous.

{\rm (ii)} Suppose that $f$
is locally paraconcave around $\bar x$. Then $f_-''(\bar x, p,
\cdot)$ is also locally paraconcave around $\bar x$.

{\rm (iii)} Suppose that $f$
is paraconcave. Then $f_-''(\bar x, p,
\cdot)$ is also paraconcave.
\end{pro}
\begin{proof}
(i) Let $u\in H$ and take any $u_n\rightarrow u$. For each $n\in\mathbb{N}$, by the definition of $f_-''(\bar x, p, u_n)$, there exist $w_n\in B(u_n, \frac{1}{n})$ and $t_n\in (0, \frac{1}{n})$ such that
$$
\frac{2(f(\bar x+t_nw_n)-f(\bar x)-t_n\langle p, w_n\rangle)}{t^2_n}-\frac{1}{n}<f_-''(\bar x, p, u_n).
$$
This implies that
\begin{eqnarray*}
\liminf_{n\rightarrow \infty}f_-''(\bar x, p, u_n)&\geq&\liminf_{n\rightarrow\infty}\big(\frac{2(f(\bar x+t_nw_n)-f(\bar x)-t_n\langle p, w_n\rangle)}{t^2_n}-\frac{1}{n}\big)\\
&\geq&\liminf_{u'\rightarrow u, t\downarrow 0}\frac{2(f(\bar x+tu')-f(\bar x)-t\langle p, u'\rangle)}{t^2}\\
&=&f_-''(\bar x, p, u).
\end{eqnarray*}
Hence $f_-''(\bar x, p, \cdot)$ is lower semicontinuous at $u$.

(ii) Suppose that there exist $\delta, \lambda>0$ such that $f-\frac{1}{2\lambda}\|\cdot\|^2$ is concave on $B(\bar x, \delta)$. We next prove that $f_-''(\bar x, p, \cdot)-\frac{1}{\lambda}\|\cdot\|^2$ is concave on $B(\bar x, \delta)$.

Let $u, v\in B(\bar x, \delta)$ and $\mu\in [0, 1]$, and take any $z\rightarrow \mu u+(1-\mu)v$ and $t\rightarrow 0$. Define $u':=u+z-(\mu u+(1-\mu)v)$ and $v':=v+z-(\mu u+(1-\mu)v)$. Then
$$z=\mu u'+(1-\mu)v'\ \ {\rm and}\ \ (u', v')\rightarrow (u, v)\ \ {\rm as}\ \ z\rightarrow \mu u+(1-\mu)v.
$$
Note that
\begin{equation*}
\begin{array}l
 \ \ \ \frac{2(f(\bar x+tz)-f(\bar x)-t\langle p, z\rangle)}{t^2}-\frac{1}{\lambda}\|\mu u+(1-\mu)v\|^2\\
=\frac{2(f(\bar x+tz)-\frac{1}{2\lambda}\|\bar x+tz\|^2-f(\bar x)-t\langle p, z\rangle+\frac{1}{2\lambda}\|\bar x+tz\|^2)}{t^2}-\frac{1}{\lambda}\|\mu u+(1-\mu)v\|^2 \\
\geq\frac{2\mu(f(\bar x+tu')-\frac{1}{2\lambda}\|\bar x+tu'\|^2)}{t^2}+\frac{2(1-\mu)(f(\bar x+tv')-\frac{1}{2\lambda}\|\bar x+tv'\|^2)}{t^2}+\\
 \ \ \ \frac{\frac{1}{2\lambda}\|\bar x+t(\mu u'+(1-\mu)v')\|^2-f(\bar x)-t\langle p, z\rangle}{t^2}-\frac{1}{\lambda}\|\mu u+(1-\mu)v\|^2\\
=\frac{2\mu(f(\bar x+tu')-f(\bar x)-t\langle p, u'\rangle)}{t^2}+\frac{2(1-\mu)(f(\bar x+tv')-f(\bar x)-t\langle p, v'\rangle)}{t^2}-\\
 \ \ \ \frac{\frac{\mu(1-\mu)}{2\lambda}\|\bar x+tu'-(\bar x+tv')\|^2}{t^2}-\frac{1}{\lambda}\|\mu u+(1-\mu)v\|^2\\
=\frac{2\mu(f(\bar x+tu')-f(\bar x)-t\langle p, u'\rangle)}{t^2}+\frac{2(1-\mu)(f(\bar x+tv')-f(\bar x)-t\langle p, v'\rangle)}{t^2}-\\
 \ \ \ \frac{\mu(1-\mu)}{2\lambda}\|u'-v'\|^2-\frac{1}{\lambda}\|\mu u+(1-\mu)v\|^2.
\end{array}
\end{equation*}
By taking lower limits, one has
\begin{eqnarray*}
&&\liminf_{z\rightarrow\mu u+(1-\mu)v, t\downarrow 0}\frac{2(f(\bar x+tz)-f(\bar x)-t\langle p, z\rangle)}{t^2}-\frac{1}{\lambda}\|\mu u+(1-\mu)v\|^2\\
&\geq&\mu\liminf_{u'\rightarrow u, t\downarrow 0}\frac{2\big(f(\bar x+tu')-f(\bar x)-t\langle p, u'\rangle\big)}{t^2}+\\
&&(1-\mu)\liminf_{u'\rightarrow u, t\downarrow 0}\frac{2\big(f(\bar x+tv')-f(\bar x)-t\langle p, v'\rangle\big)}{t^2}-\\
&&\frac{\mu(1-\mu)}{\lambda}\|u-v\|^2
-\frac{1}{\lambda}\|\mu u+(1-\mu)v\|^2\\
&\geq&\mu(f_-''(\bar x, p, u)-\frac{1}{\lambda}\|u\|^2)+(1-\mu)(f_-''(\bar x, p, v)-\frac{1}{\lambda}\|v\|^2)\\
&&+\frac{\mu}{\lambda}\|u\|^2+\frac{1-\mu}{\lambda}\|v\|^2-\frac{\mu(1-\mu)}{\lambda}\|u-v\|^2
-\frac{1}{\lambda}\|\mu u+(1-\mu)v\|^2\\
&=&\mu(f_-''(\bar x, p, u)-\frac{1}{\lambda}\|u\|^2)+(1-\mu)(f_-''(\bar x, p, v)-\frac{1}{\lambda}\|v\|^2).
\end{eqnarray*}
This implies that  $f_-''(\bar x, p, \cdot)-\frac{1}{\lambda}\|\cdot\|^2$ is concave on $B(\bar x, \delta)$.

Note that (iii) follows from (ii) and thu the proof is completed.
\end{proof}

For a twice epi-differentiable function $f$ defined on a finite-dimensional space, it is shown in \cite[Corollary 8.47]{19} that second-order epi-derivative of $f$ closely relates to protoderivative of limiting subdifferential $\partial f$. The authors \cite{28} also studied the relationship between the second-order epi-derivative and the protoderivative in a Hilbert space. In order to study the application of  twice epi-differentiability in this paper, we first recall the following two important concepts of set convergence.

For a sequence $\{C_n: n\in \mathbb{N}\}$
of closed subsets in $H$,
$\liminf_{n\rightarrow\infty}C_n$ denotes the set of all limit points of sequences $\{x_n\}$ with $x_n\in C_n$ for all $n\in\mathbb{N}$, and $\limsup_{n\rightarrow\infty}C_n$ denotes the set of all cluster points of such sequences.
Recall that $\{C_n: n\in \mathbb{N}\}$ is said to be {\it Painlev\'{e}-Kuratowski convergent} to a subset $C$ of $H$, if
$$
C=\liminf\limits_{n\rightarrow\infty}C_n=\limsup\limits_{n\rightarrow\infty}C_n.
$$
and that $\{C_n: n\in \mathbb{N}\}$ is said to be {\it Mosco convergent} to a subset $C$ of $H$, if
$$
C=\liminf\limits_{n\rightarrow\infty}C_n={\rm w}\mbox{-}\limsup\limits_{n\rightarrow\infty}C_n,
$$
where ${\rm w}\mbox{-}\limsup\limits_{n\rightarrow \infty}C_n$ is  the set of all weak cluster points of sequences from the sets $C_n$, that is, $x\in {\rm w}\mbox{-}\limsup\limits_{n\rightarrow \infty}C_n$ if and only if there exists a sequence $\{x_n\}$ such that $x_n\in C_n$ for all $n\in\mathbb{N}$ and a subsequence of $\{x_n\}$ converges to $x$ with respect to the weak topology.

Let $f, f_n: H\rightarrow \mathbb{R}\cup\{+\infty\}(n=1,2,\cdots)$ be proper
lower semicontinuous functions. We say that $\{f_n\}$ is {\it Mosco} (resp.{\it Painlev\'{e}-Kuratowski}) epi-convergent to $f$, if ${\rm epi}(f_n)$ is Mosco (resp. Painlev\'{e}-Kuratowski) convergent to ${\rm epi}(f)$; in the Mosco epi-convergent (resp. Painlev\'{e}-Kuratowski epi-convergent) case we write
$$
f={\rm M}\mbox{-}\lim_{n\rightarrow \infty}f_n ({\rm resp.}\ f={\rm PK}\mbox{-}\lim_{n\rightarrow \infty}f_n).
$$

Recall that $f: H\rightarrow \mathbb{R}\cup\{+\infty\}$ is said to be {\it twice epi-differentiable} at $\bar x$ relative to $p\in \partial_pf(\bar x)$ in the sense of Mosco (resp. Painlev\'{e}-Kuratowski), if the second-order difference quotient functions $\Delta_2f(\bar x, p, t, \cdot)$ are Mosco epi-convergent (resp. Painlev\'{e}-Kuratowski epi-convergent) to a proper function as $t\rightarrow 0^+$; that is $f$ is  twice epi-differentiable at $\bar x$ relative to $p\in\partial_p f(\bar x)$ in the sense of Mosco (resp. Painlev\'{e}-Kuratowski) if and only if for any sequence $\{t_n\}$ in $(0, +\infty)$ convergent to 0, the function sequence $\{\Delta_2f(\bar x, p, t_n,
\cdot)\}$ is Mosco epi-convergent (resp. Painlev\'{e}-Kuratowski epi-convergent) to the same proper function. The Mosco epi-limit of these second-order difference quotient functions is called {\it second-order epi-derivative} of $f$ at $\bar x$ relative to $p$ and is denoted by $f_-''(\bar x; p)(\cdot)$. In this case, one can easily verify that
\begin{equation}\label{3.8a}
f_{-}''(\bar x;p)(h)=f_{-}''(\bar x,p, h)\ \ \forall h\in H.
\end{equation}

\noindent{\bf Remark 3.1.} From the definition, it is known that twice epi-differentiability of a function in the sense of Mosco is stronger than that in the sense of Painlev\'{e}-Kuratowski in the Hilbert space. When the space is finite-dimensional, both concepts coincide and reduce to the corresponding notion of twice epi-differentiability (cf. \cite[Definition 22]{11}). Unless otherwise stated, the twice epi-differentiability of a function studied in this paper is in the sense of Mosco.\\

The following theorem is a key tool in proving main results of this paper. Readers are invited to consult \cite[Theorem 4.2]{28} for more details and its proof.\\

\noindent{\bf Theorem A.} {\it Let $f, f_n: H\rightarrow \mathbb{R}\cup\{+\infty\}(n=1,2,\cdots)$ be proper
lower semicontinuous functions such that $f={\rm M}\mbox{-}\lim\limits_{n\rightarrow \infty}f_n$. Then for any $p\in \partial f(x)$ there exist sequences $\{(x_n, p_n)\}$ in $H\times H$ and a strictly increasing sequence $\{n_k\}$ in $\mathbb{N}$ such that}
$$
(x_{n_k}, f_{n_k}(x_{n_k}))\longrightarrow (x, f(x)), \ \ p_{n_k}\stackrel{w}\longrightarrow p\ \ and\ \  p_{n_k}\in \partial_pf_{n_k}(x_{n_k})\ {\it for\ all} \ k\in\mathbb{N}.
$$

The following proposition refers to epi-convergence of functions and epigraph of second-order lower Dini-directional derivative.

\begin{pro}
Let $f:H\rightarrow \mathbb{R}\cup\{+\infty\}$ be a proper lower
semicontinuous function and $\bar x\in{\rm dom}f$ with $p\in\partial_pf(\bar x)$. Then
\begin{equation}\label{3.9a}
\liminf_{t\rightarrow 0^+}{\rm epi}\big(\frac{1}{2}\Delta_2f(\bar x,
p, t, \cdot)\big)\subset {\rm epi}\big(\frac{1}{2}f_-''(\bar x, p, \cdot)\big)
\end{equation}
and
\begin{equation}\label{3.9b}
{\rm epi}\big(\frac{1}{2}f_-''(\bar x, p, \cdot)\big)\subset \limsup_{t\rightarrow 0^+}{\rm epi}\big(\frac{1}{2}\Delta_2f(\bar x,
p, t, \cdot)\big).
\end{equation}
\end{pro}
\begin{proof}
Let $(h, r)\in\liminf\limits_{t\rightarrow 0^+}{\rm epi}\big(\frac{1}{2}\Delta_2f(\bar x,
p, t, \cdot)\big)$. Then for any $t_k\rightarrow 0^+$, there exists $(h_k, r_k)\in {\rm epi}\big(\frac{1}{2}\Delta_2f(\bar x,
p, t_k, \cdot)\big) (k\in \mathbb{N})$ such that $(h_k, r_k)\rightarrow (h, r)$. This implies that
$$
\frac{1}{2} f_-''(\bar x,
p, h)\leq\liminf_{k\rightarrow \infty}\frac{1}{2}\Delta_2f(\bar x,
p, t_k, h_k)\leq\liminf_{k\rightarrow \infty}r_k=r.
$$
Hence $(h, r)\in{\rm epi}\big(\frac{1}{2}f_-''(\bar x, p, \cdot)\big)$ and consequently \eqref{3.9a} holds.

Next, let $(h, r)\in{\rm epi}\big(\frac{1}{2}f_-''(\bar x, p, \cdot)\big)$. By \eqref{3.5}, there are sequences $h_k\rightarrow h$ and $t_k\rightarrow 0^+$ such that $\frac{1}{2}\Delta_2f(\bar x,
p, t_k, h_k)\rightarrow \frac{1}{2}f_-''(\bar x, p, h)$. For each $k\in \mathbb{N}$, let
$$
r_k:=\frac{1}{2}\Delta_2f(\bar x,
p, t_k, h_k)+r-\frac{1}{2}f_-''(\bar x, p, h).
$$
Then $(h_k, r_k)\in {\rm epi}\big(\frac{1}{2}\Delta_2f(\bar x,
p, t_k, \cdot)\big)$ and $(h_k, r_k, t_k)\rightarrow (h, r, 0^+)$. From this, one has
$$
(h, r)\in \limsup_{t\rightarrow 0^+}{\rm epi}\big(\frac{1}{2}\Delta_2f(\bar x,
p, t, \cdot)\big).
$$
Thus \eqref{3.9b} holds. The proof is completed.
\end{proof}

\setcounter{equation}{0}

\section{Second-order optimality conditions}
This section is devoted to the study of second-order optimality conditions defined by three generalized second-order derivatives in the Hilbert space and the equivalence interrelationship among them. We begin with three types of second-order optimality conditions of an extended real-valued nonsmooth function on finite-dimension space studied by Eberhard and Wenczel \cite{11} (cf. \cite[Definition 56]{11} and \cite[Definition 6.1]{EM}).\\

\noindent{\bf Definition 4.1.} {\it Let $f:\mathbb{R}^m\rightarrow \mathbb{R}\cup\{+\infty\}$ be a proper lower semicontinuous function and assume that the first-order condition $0\in\partial_pf(\bar x)$ holds.

1. We say that $f$ satisfies the second-order condition of the first kind at $\bar x$, if there exists $\beta\in(0,+ \infty)$ such that $f_-''(\bar x, 0, h)\geq\beta$ for all $h\in S_{\mathbb{R}^m}$.

2. We say that  $f$ satisfies the second-order condition of the second kind at $\bar x$, if there exists $\beta\in(0,+ \infty)$ satisfying for all $h\in{\rm dom}D^2f(\bar x, 0)\cap S_{\mathbb{R}^m}$, there is $z\in D^2f(\bar x, 0)(h)$ such that $\langle z, h\rangle\geq\beta$.

3. We say that  $f$ satisfies the second-order condition of the third kind at $\bar x$, if there exists $\beta\in(0,+ \infty)$ such that for any $h\in S_{\mathbb{R}^m}$ and any $z\in\hat{\partial}^2f(\bar x, 0)(h)$, one has $\langle z, h\rangle\geq\beta$.}\\

It is known that Eberhard and Wenczel \cite{11} mainly investigate the close interrelationship among these optimality conditions and proved the following result on the equivalence among these optimality conditions in Definition 4.1 for paraconcave functions (cf. \cite[Theorem 6.3]{EM} and \cite[Theorem 66]{11}).\\

\noindent{\bf Theorem B.} {\it Let $f:\mathbb{R}^m\rightarrow \mathbb{R}\cup\{+\infty\}$ be a prox-bounded and lower semicontinuous function with $0\in\partial_pf(\bar x)$. Suppose that $f$ is finite and there exists $c>0$ such that $f-\frac{c}{2}\|\cdot\|^2$ and $f_-''(\bar x, 0, \cdot)-c\|\cdot\|^2$ are concave. Then all second-order optimality conditions are equivalent. Moreover, the same $\beta$ value may be used in each condition.}\\

As one part of main work in this paper, it is natural to study the original forms of second-order optimality conditions in the Hilbert space. However, the existing implication among these optimality conditions given in Definition 4.1 may not be valid for the case of Hilbert space (comparing Proposition 4.2 below with \cite[Proposition 45]{11}), and thus it is necessary to make some minor modification to these optimality conditions. Motivated by this observation, we consider the following second-order optimality conditions in the Hilbert space.\\

\noindent{\bf Definition 4.2.} {\it Let $f:H\rightarrow \mathbb{R}\cup\{+\infty\}$ be a lower semicontinuous function and assume that the first-order optimality
condition $0\in\partial_pf(\bar x)$ holds.

{\rm (i)} We say that  $f$ satisfies the second-order optimality condition of the first kind at $\bar x$, if there exists $\beta>0$ such that $ f_-''(\bar x, 0, h)\geq \beta$  for all $h\in S_H$.

{\rm (ii)} We say that $f$ satisfies the second-order optimality condition of the second kind at $\bar x$, if there exists $\beta>0$ such that for all $h\in{\rm dom}D_M^2f(\bar x, 0)\cap S_H$, there is $z\in D^2_Mf(\bar x, 0)(h)$ such that $
\langle z, h\rangle\geq\beta$.

{\rm (iii)} We say that $f$ satisfies the second-order optimality condition of the third kind at $\bar x$, if there exists $\beta>0$ such that for all $h\in S_H\cap{\rm dom}\partial^2_Mf(\bar x, 0)$ and $z\in \partial_M^2f(\bar x, 0)(h)$, one has $\langle z, h\rangle\geq \beta$.}\\

\noindent{\bf{Remark 4.1.}} Note that mixed contingent cone and contingent cone coincide in the finite-dimensional space setting. Hence when restricted to the finite-dimensional space, second-order optimality conditions in Definition 4.2 reduce to those studied in \cite{11,EM} as Definition 4.1.\\

Now, we pay main attention to the equivalence interrelationship among these second-order optimality conditions in Definition 4.2. We first provide the following proposition whose proof mainly relies on Theorem A aforementioned in Section 3. This proposition is one key tool to prove main results in this section.

\begin{pro}
Let $f:H\rightarrow\mathbb{R}\cup\{+\infty\}$ be a proper lower semicontinuous function
and $\bar x\in{\rm dom}f$ with $p\in \partial_pf(\bar x)$. Suppose that $f$ is twice
epi-differentiable at $\bar x$ for $p$. Then
\begin{equation}\label{4.1a}
  \frac{1}{2}\partial f''_-(\bar x, p, \cdot)(h)\subset D_M^2f(\bar x, p)(h)\ \ \forall \ h\in H
\end{equation}
and
\begin{equation}\label{4.11}
f_-''(\bar x, p, h)\leq \sup\Big\{\langle w, h\rangle: w\in D^2_Mf(\bar x, p)(h)\Big\}
\end{equation}
holds for all $h\in{\rm dom}\partial f''_-(\bar x, p, \cdot)$.

Assume further that $f$ is a paraconcave function and continuous at $\bar x$. Then
\begin{equation}\label{4-2a}
  {\rm dom} D^2_Mf(\bar x, p)={\rm dom}\partial f_-''(\bar x, p, \cdot)=H
\end{equation}
and
\begin{equation}\label{4.12}
f_-''(\bar x, p, h)\leq \sup\Big\{\langle w, h\rangle: w\in D^2_Mf(\bar x, p)(h)\Big\} \ \ \forall h\in H.
\end{equation}
\end{pro}

\begin{proof}
Let $h\in{\rm dom}\partial f''_-(\bar x, p, \cdot)$ and $z\in \frac{1}{2}\partial f''_-(\bar x, p, \cdot)(h)$. Since $f$ is twice epi-differentiable at $\bar x$ for $p$, by \eqref{3.8a} and Theorem A, there exist sequences $t_n\rightarrow 0^+$, $h_n\rightarrow h$ and
$z_n\stackrel{w}\longrightarrow z$ such that
$$
\Delta_2f(\bar x, p, t_n,
h_n)\rightarrow
f_-''(\bar x, p, h)\ {\rm and}\ 2z_n\in\partial_p \Delta_2f(\bar x, p, t_n,
\cdot)(h_n).
$$
By virtue of \eqref{3.6}, one has
$$
z_n\in\partial_p\big(\frac{1}{2}\Delta_2f(\bar x, t_n, p, \cdot)\big)(h_n)=\frac{1}{t_n}(\partial_pf(\bar x+t_nh_n)-p).
$$
This implies that $(x+t_nh_n, p+t_nz_n)\in{\rm gph}(\partial_pf)$ and consequently it follows from \eqref{3.1} that $(h, z)\in T_M({\rm gph}(\partial_pf), (\bar x, p))$. Thus $z\in D_M^2f(\bar x,p)(h)$ and \eqref{4.1a} holds.

Noting that $f''_-(\bar x, p, \cdot)$ is 2-positively homogeneous and $2z\in \partial f''_-(\bar x, p, \cdot)(h)$, it follows from \cite[Theorem 3.1]{YW1} that $f''_-(\bar x, p, h)=\langle z, h\rangle $ and thus \eqref{4.11} holds.

Assume that $f$ is a paraconcave function and continuous at $\bar x$. We can take $\lambda>0$
such that $g(u):=f(u)-\frac{1}{2\lambda}\|u\|^2$ is concave. Using Proposition 3.1 and Corollary 3.1, one has that $f_-''(\bar x, p, \cdot)-\frac{1}{\lambda}\|\cdot\|^2$ is also concave. By computing, one has ${\rm dom}\partial_pf_-''(\bar
x, p, \cdot)={\rm dom}\partial_pg_-''(\bar x, q, \cdot)$ and
$$
g_-''(\bar x, q, u)=f_-''(\bar x, p, u)-\frac{1}{\lambda}\|u\|^2 \ \ \forall u\in H,
$$
where $q:=p-\frac{1}{\lambda}\bar x\in \partial_pg(\bar x)$. Thus, $g_-''(\bar x, q, \cdot)$ is concave. We claim that
\begin{eqnarray}\label{4.13}
{\rm dom}f_-''(\bar x, p, \cdot)={\rm dom}g_-''(\bar x, q, \cdot)=H.
\end{eqnarray}

Let $h\in H$.  From $q\in \partial_pg(\bar x)$, there exist $r_0,\delta_0>0$ such that
\begin{equation}\label{4-5}
  g(x)\geq g(\bar x)+\langle q, x-\bar x\rangle-\frac{r_0}{2}\|x-\bar x\|^2\ \ \forall x\in B(\bar x, \delta_0).
\end{equation}
Since $g$ is concave and $f$ is continuous at $\bar x$, it follows that $-g$ is continuous at $\bar x$ and $\partial_p(-g)(\bar x)\not=\emptyset$. Then we can choose $\zeta\in\partial_p(-g)(\bar x)$ and it follows from the convexity of $-g$ that
\begin{equation}\label{4-6}
  \langle-\zeta,x-\bar x\rangle\geq g(x)-g(\bar x)\ \ \forall x\in H.
\end{equation}
By \eqref{4-5} and \eqref{4-6}, for any $v\in H$ and any $t>0$ sufficiently small, one has
\begin{eqnarray*}
\langle-\zeta, tv\rangle\geq g(\bar x+tv)-g(\bar x)\geq \langle q, tv\rangle-\frac{r_0}{2}\|tv\|^2.
\end{eqnarray*}
From this, one can verify that $q=-\zeta$ and thus
$$
\langle q, tv\rangle\geq g(\bar x+tv)-g(\bar x)\ \ \forall t>0\ {\rm and}\ \forall v\in H.
$$
This implies that
\begin{eqnarray}\label{4.14}
g_-''(\bar x, q, h)=\liminf_{h'\rightarrow h, t\downarrow 0}\frac{2(g(\bar x+th')-g(\bar x)-\langle q, th'\rangle)}{t^2}\leq 0.
\end{eqnarray}
Noting that $p\in \partial_pf(\bar x)$, there are $r, \delta>0$ such that
$$
f(x)\geq f(\bar x)+\langle p, x-\bar x\rangle-\frac{r}{2}\|x-\bar x\|^2\,\,\,\,\forall x\in B(\bar x, \delta).
$$
Thus,
\begin{eqnarray}\label{4.15}
\ \ \ \ f_-''(\bar x, p, h)=\liminf_{h'\rightarrow h, t\downarrow 0}\frac{2(f(\bar x+th')-f(\bar x)-\langle p, th'\rangle)}{t^2}\geq -r\|h\|^2.
\end{eqnarray}
Since $g_-''(\bar x, q, h)=f_-''(\bar x, p, h)-\frac{1}{\lambda}\|h\|^2$, it follows from \eqref{4.14} and \eqref{4.15} that
$$
-\infty< f_-''(\bar x, p, h) <+\infty,
$$
and consequently $g_-''(\bar x, q, h)\in \mathbb{R}$. Hence \eqref{4.13} holds.

We next prove that
\begin{equation}\label{4.9c}
 {\rm dom}\partial f_-''(\bar x, p, \cdot)=H.
\end{equation}
Granting this, it follows that \eqref{4-2a} holds and \eqref{4.12} holds by \eqref{4.11}.


Let $h\in {\rm dom}\partial f_-''(\bar x, p, \cdot)$. Since $f_-''(\bar x, p, \cdot)$ is lower semicontinuous (by Proposition 3.1) and $h\in H={\rm dom}f_-''(\bar x, p, \cdot)$, by virtue of Density Theorem (cf. \cite[Theorem 3.1]{8}), there exists $(h_k, z_k)\in H\times H$ such that $2z_k\in \partial_pf_-''(\bar x, p, h_k)$ and $h_k\rightarrow h$. Now, using the concavity of $g_-''(\bar x, q, \cdot)$, one has that $g_-''(\bar x, q, \cdot)$ is locally Lipschtzian, and so is $f_-''(\bar x, p, \cdot)$. Then, by virtue of \cite[Theorem 7.3]{8}, there exists $L:=L(h)>0$ such that when $k$ is sufficiently large, one has $\|z_k\|\leq L$. By applying \cite[Corollary 2.8.9]{15}, $\{z_k\}$ has a weakly convergent subsequence $\{z_{k_i}\}$ and so we can assume that $z_{k_i}\stackrel{w}\longrightarrow z\in H$ as $i\rightarrow \infty$. This implies that $2z\in \partial f_-''(\bar x, p, \cdot)(h)$ as $h_{k_i}\rightarrow h (i\rightarrow \infty)$ and $2z_{k_i}\in \partial_pf_-''(\bar x, p, \cdot)(h_{k_i})$. Hence $h\in {\rm dom}\partial f_-''(\bar x, p, \cdot)$. The proof is completed.
\end{proof}

\noindent{\bf{Remark 4.2}} The proof of Proposition
4.1 is inspired by the idea from the proof of \cite[Proposition 45]{11} and Proposition 4.1 is an extension and improvement of \cite[Proposition 45]{11} since for the case of twice epi-differentiable function, the assumption that ``{\it$h\mapsto f_-''(\bar x, p, h)$ is paraconcave}" is dropped from Proposition 4.1 and the conclusion ``{\it ${\rm dom}D_M^2f(\bar x, p)=H$}" is stronger than that ``{\it${\rm dom}D^2f(\bar x, p)$ is dense in $\mathbb{R}^m$}" in \cite[Proposition 45]{11} ($D_M^2f(\bar x, p)$ coincides with $D^2f(\bar x, p)$ in $\mathbb{R}^m$).\\

The following proposition follows from Proposition 4.1.

\begin{pro}
Let $f:H\rightarrow \mathbb{R}\cup\{+\infty\}$ be a lower
semicontinuous function and $\bar x\in{\rm dom}f$ with $p\in \partial_pf(\bar x)$. Suppose that $f$ is twice epi-differentiable at $\bar x$ for $p$. Then
\begin{equation}\label{4.16}
\sup\big\{\langle w, h\rangle: w\in \partial^2_Mf(\bar x, p)(h)\big\}\leq f_-''(\bar x, p, h)
\end{equation}
and
\begin{equation}\label{4.10a}
\partial^2_Mf(\bar x, p)(h)\subset\{w\in H: \langle w, h\rangle\leq f_-''(\bar x, p, h)\}
\end{equation}
hold for all $h\in
{\rm dom}\partial f_-''(\bar x, p, \cdot)$.
\end{pro}
\begin{proof} Let $h\in {\rm dom}\partial f_-''(\bar x, p, \cdot)$ and choose $z\in H$ such that $z\in \frac{1}{2}\partial f_-''(\bar x, p, \cdot)(h)$. Then, by \cite[Theorem 3.1]{YW1}, one has
\begin{equation}\label{4.11a}
  f_-''(\bar x, p, h)=\langle h, z\rangle.
\end{equation}
Let $w\in \partial^2_Mf(\bar x, p)(h)$. Then
\begin{eqnarray}\label{4.12a}
\langle w, y\rangle\leq\langle h, v\rangle\ \ \forall (y, v)\in T_M({\rm gph}(\partial_pf), (\bar x, p)).
\end{eqnarray}
Note that $z\in \frac{1}{2}\partial f_-''(\bar x, p, \cdot)(h)$ and thus $z\in D^2_Mf(\bar x, p)(h)$ by Proposition 4.1. Using \eqref{4.11a} and \eqref{4.12a}, one has
$$
\langle w, h\rangle\leq\langle h, z\rangle=f_-''(\bar x,p,h).
$$
This means that \eqref{4.16} holds and so does \eqref{4.10a}. The proof is completed.
\end{proof}

Combining Proposition 4.1 with Proposition 4.2, we have the following theorem.
\begin{them}
Let $f:H\rightarrow \mathbb{R}\cup\{+\infty\}$ be a proper lower
semicontinuous and paraconcave function and $\bar x\in{\rm dom}f$ with $0\in \partial_pf(\bar x)$. Suppose that $f$ is continuous at $\bar x$ and twice epi-differentiable at $\bar x$ for $0$. Then the second-order optimality condition of the first kind implies that of the second kind. Furthermore, if  $f$ satisfies the second-order
optimality condition of the third kind at $\bar x$, then there exists $\beta>0$ such that
\begin{eqnarray}\label{4.17}
f_-''(\bar x, 0, h)\geq \beta
\end{eqnarray}
holds  for all $h\in S_H\cap {\rm dom}\partial^2_Mf(\bar x, 0)$.
\end{them}
\begin{proof}
Suppose that there exists $\beta>0$ such that the second-order optimality condition of the first kind holds. Let $\beta_1\in (0, \beta)$ and $h\in {\rm dom}D^2_Mf(\bar x, 0)\cap S_H$. Since $f$ is paraconcave, it follows from Proposition 4.1 and optimality condition of the first kind that
$$
\sup\Big\{\langle w, h\rangle: w\in D^2_Mf(\bar x, 0)(h)\Big\}\geq f_-''(\bar x, 0, h)\geq \beta>\beta_1.
$$
Then there exists $z\in D^2_M(\partial_pf)(\bar x, 0)(h)$ such that $\langle z, h\rangle\geq\beta_1$. Hence $f$ satisfies the second-order optimality condition of the second kind at $\bar x$ with constant $\beta_1$.

Suppose that there exists $\beta>0$ such that the second-order
optimality condition of the third kind holds. Let $h\in S_H\cap {\rm dom}\partial^2_Mf(\bar x, 0)$. Since $f$ is paraconcave, by Proposition 4.1, one has $h\in {\rm dom}\partial f_-''(\bar x, p, \cdot)$ and it follows from Proposition 4.2 and the second-order optimality condition of the third kind that
$$
f_-''(\bar x, 0, h)\geq \sup\big\{\langle w, h\rangle: w\in \partial^2_Mf(\bar x, 0)(h)\big\}\geq\beta.
$$
The proof is completed.
\end{proof}

Next, we study the duality between second-order optimality conditions of the second and the third kinds. In order to deal with
it, we denote the following linear mapping by ${\rm Proj}_h(h',
z'):=h'$ and ${\rm Proj}_z(h', z'):=z'$. For a subset $S\subset
H\times H$, let
$$
{\rm Proj}_hS:=\{h\in H: \exists (h, z)\in S\} \ \ {\rm and} \ \
{\rm Proj}_zS:=\{z\in H: \exists (h, z)\in S\}.
$$
Then, when $0\in
\partial_pf(\bar x)$, we have
$$
{\rm dom}\partial^2_Mf(\bar x, 0)=-{\rm
Proj}_z N_M({\rm gph}(\partial f), (\bar x, 0))
$$
and
$$
{\rm dom}D^2_Mf(\bar x, 0)={\rm Proj}_hT_M({\rm gph}(\partial
f), (\bar x, 0)).
$$

\begin{pro} Let $T$ be a closed convex cone in $H\times H$ and $T(h):=\{z\in H: (h, z)\in T\}$, and let $N:=T^{\circ}$.
Consider the following statements:

{\rm (i)} there exists $\beta>0$ such that  $\langle h, z\rangle\leq -\beta$ for all $(h, z)\in N$ with $\|z\|=1$;

{\rm (ii)} there exists $\beta_1>0$ such that for each $h\in S_H$ with $h\in {\rm sqri(Proj}_hT)$, there exists $(h, z)\in T$ such that $\langle h, z\rangle\geq \beta_1$.

Then ${\rm (i)}\Rightarrow{\rm (ii)}$. Furthermore, we assume that ${\rm qri}N\not=\emptyset$ and
\begin{equation}\label{4.18}
{\rm qri}({\rm Proj}_zN)\subset{\rm sqri(Proj}_hT).
\end{equation}
Then ${\rm (ii)}\Rightarrow{\rm (i)}$.
\end{pro}
\begin{proof} ${\rm (i)}\Rightarrow{\rm (ii)}$:
Let $\beta_1\in (0, \beta)$ and $h\in S_H\cap {\rm sqri(Proj}_hT)$. Define $v(h):=\sup\{\langle h, z\rangle: (h, z)\in T\}$. If $v(h)=+\infty$, then the conclusion holds. Next, we assume that $v(h)<+\infty$. Then, by computing, we have
\begin{equation}\label{4.19}
-v(h)=\inf_{(\hat h, \hat z)}\{\delta_T(\hat h, \hat z)+(\delta_{\{h\}}(\hat h)-\langle \hat h, \hat z\rangle)\}.
\end{equation}
Let $f(\hat h, \hat z):=\delta_T(\hat h, \hat z)$ and $g(\hat h, \hat z):=\delta_{\{h\}}(\hat h)-\langle \hat h, \hat z\rangle$. Then it is easy to verify that $g$ is convex. Noting that $h\in{\rm sqri(Proj}_hT)$, it follows that
$$
(0, 0)\in{\rm sqri(dom}f-{\rm dom}g).
$$
Applying the Fenchel duality in infinite-dimensional spaces(cf. \cite{3, 20}), one has
\begin{eqnarray*}
-v(h)&=&\sup_{(h^*, z^*)}\{-g^*(h^*, z^*)-f^*(-h^*, -z^*))\}\\
&=&\sup_{(h^*, z^*)}\{-\langle h, h^*\rangle-\delta_{\{-h\}}(z^*)-\delta_{T^{\circ}}(-h^*, -z^*)\}\\
&=&\sup_{h^*}\{-\langle h, h^*\rangle-\delta_{T^{\circ}}(-h^*, h)\}.
\end{eqnarray*}
Hence
$$
v(h)=\inf_{h^*}\{\langle h, h^*\rangle+\delta_{T^{\circ}}(-h^*, h)\}=\inf_{\{h^*: (h^*, h)\in T^{\circ}\}}\{\langle h, -h^*\rangle\}.
$$
Since $v(h)<+\infty$, then $\{h^*: (h^*, h)\in T^{\circ}\}\not=\emptyset$. For each  $h^*\in H$ with $(h^*, h)\in T^{\circ}=N$, by the assumption, one has $\langle h, -h^*\rangle\geq \beta$ and consequently $v(h)\geq \beta>\beta_1$. Thus there exists $z\in T(h)$ such that $\langle h, z\rangle\geq \beta_1$.

${\rm (ii)}\Rightarrow{\rm (i)}$: Let $(h^*, h)\in N$ with $\|h\|=1$. By \cite[Proposition 2.5]{3}, there exists $(h^*_n, h_n)\in {\rm qri}N$ such that $(h^*_n, h_n)\rightarrow (h^*, h)$ and $\|h^*_n\|\rightarrow \|h\|=1$. Let $\hat{h}_n:=\frac{h_n}{\|h_n\|}$. Noting that $h_n\in{\rm qri}({\rm Proj}_zN)$ and $0\in {\rm Proj}_zN$, it follows that $\hat h_n\in {\rm qri}({\rm Proj}_zN)$ as $N$ and ${\rm Proj}_zN$ are cones. By the assumption, we have $\hat h_n\in {\rm sqri(Proj}_hT)$. Applying the Fenchel duality again, one has
\begin{equation}\label{4.20}
\inf_{\{h^*: (h^*, \hat h_n)\in T^{\circ}\}}\langle -h^*, \hat h_n\rangle=v(\hat h_n)\geq \beta_1.
\end{equation}
Since $(\frac{h^*_n}{\|h_n\|}, \hat h_n)\in N$ (thanks to $(h^*_n, h_n)\in N$), it follows from \eqref{4.20} that
$$
\langle -\frac{h^*_n}{\|h_n\|}, \hat h_n\rangle\geq\inf_{\{h^*: (h^*, \hat h_n)\in T^{\circ}\}}\langle -h^*, \hat h_n\rangle\geq\beta_1.
$$
Taking limits as $n\rightarrow \infty$, we have $\langle -h^*, h\rangle\geq\beta_1$. Hence $\langle h^*, h\rangle\leq-\beta_1$ and consequently (i) holds for $\beta=\beta_1$. The proof is completed.
\end{proof}

By using Proposition 4.3, the following theorem is immediate.
\begin{them}
Let $f:H\rightarrow \mathbb{R}\cup\{+\infty\}$ be a proper lower
semicontinuous function and $\bar x\in{\rm dom}f$ with $0\in \partial_pf(\bar x)$. Suppose that ${\rm qri}( N_M({\rm gph}(\partial_p f), (\bar x, 0)))\not=\emptyset$ and
\begin{equation}\label{4.22}
{\rm qri}\big({\rm dom}\partial^2_Mf(\bar x, 0)\big)\subset -{\rm sqri(dom}D^2_Mf(\bar x, 0)).
\end{equation}
Then the second-order optimality condition of the second kind implies that of the third kind.
\end{them}
\begin{proof}
Let $T:=\overline{\rm co}(T_M({\rm gph}(\partial_pf), (\bar x, 0)))$ and $N:=T^{\circ}$.
Then
$$N=N_M({\rm gph}(\partial_pf), (\bar x, 0)),
$$
and for each $h\in{\rm dom}D^2_Mf(\bar x, 0)$, one
has
$$
\sup\big\{\langle h, z\rangle: (h, z)\in T_M({\rm gph}(\partial_pf), (\bar x, 0))\big\}\leq v(h),
$$
where $v(h)$ is defined as in the proof of Proposition 4.3. Using the proof of (ii)$\Rightarrow$(i) in Proposition
4.3, one can prove Theorem 4.2. The proof is completed.
\end{proof}

The following theorem, as one main result in this paper, establishes the equivalence between these second-order optimality conditions for paraconcave and twice epi-differentiable functions in the Hilbert space.

\begin{them}
Let $f:H\rightarrow \mathbb{R}\cup\{+\infty\}$ be a proper lower
semicontinuous and paraconcave function and $\bar x\in{\rm dom}f$ with $0\in \partial_pf(\bar x)$. Suppose that $f$ is continuous at $\bar x$ and is twice epi-differentiable at $\bar x$ for $0$, ${\rm qri}( N_M({\rm gph}(\partial_p f), (\bar x, 0)))\not=\emptyset$ and ${\rm Proj_z}N_M({\rm gph}(\partial_p f), (\bar x, 0))=H$. Then all second-order optimality conditions are equivalent.
\end{them}
\begin{proof} By Theorem 4.1, one can get that the first kind implies the second kind.

Using Proposition 4.1, one has ${\rm dom}D^2_Mf(\bar x, 0)=H$ and thus
$${\rm sqri(dom}D^2_Mf(\bar x, 0))=H.
$$
This implies that \eqref{4.22} holds trivially and it follows from Theorem 4.2 that the second kind implies the third kind.

Since ${\rm dom}\partial_M^2f(\bar x, 0)=-{\rm Proj_z}N_M({\rm gph}(\partial_p f), (\bar x, 0))=H$, then the third kind  implies the first kind by virtue of Theorems 4.1. The proof is completed.
\end{proof}

\noindent{\bf Remark 4.3.} For twice epi-differentiable functions, Theorem 4.3 is an extension of Theorem B from the finite-dimensional space to the Hilbert space setting under some mild assumptions. When restricted to the case of  finite-dimensional spaces, the quasi-relative interior reduces to the relative interior and thus the assumption ${\rm qri}(N_M({\rm gph}(\partial_p f), (\bar x, 0)))\not=\emptyset$ holds trivially as $N_M({\rm gph}(\partial_p f), (\bar x, 0))$ is convex. With regard to assumption ${\rm Proj_z}N_M({\rm gph}(\partial_p f), (\bar x, 0))=H$, even in finite-dimensional space, very few is known about ${\rm dom}\partial_M^2f(\bar x, 0)$ and the inner estimate for ${\rm dom}\partial_M^2f(\bar x, 0)$ is currently lacking, both making analysis of optimality condition involving ${\rm dom}\partial_M^2f(\bar x, 0)$ difficult. Therefore, it is necessary to add this assumption in the analysis of second-order optimality conditions.

\setcounter{equation}{0}
\section{Applications to strict local minimizers of order two}

In this section, we apply main results on second-order optimality conditions obtained in Section 4 to strict local minimizers of order in the Hilbert space and aim to provide its necessary and/or sufficient conditions. We begin with the definition of strict local minimizer of order two.\\

\noindent{\bf Definition 5.1} {\it
Let $f: H\rightarrow\mathbb{R}\cup\{+\infty\}$ be a proper lower semicontinuous function. We say that $\bar x\in H$ is a strict local minimizer of order two for $f$, if there exist constants $\beta, \delta\in (0, +\infty)$ such that
\begin{equation}\label{5.1}
  f(x)\geq f(\bar x)+\frac{\beta}{2}\|x-\bar x\|^2
\end{equation}
holds for all $x\in B(\bar x, \delta)$.}\\

The following theorem is a known and key characterization for strict local minimizers of order two in finite-dimension space. This theorem is established via second-order lower Dini-directional derivative. Readers are invited to consult \cite[Lemma 58]{11} and \cite[Proposition 3.3]{23} for more details.\\

\noindent{\bf Theorem C.} {\it Let $f:\mathbb{R}^m\rightarrow \mathbb{R}\cup\{+\infty\}$ be a proper lower semicontinuous function and $\bar x\in {\rm dom}f$. Assume that the first-order optimality condition $0\in \partial_pf(\bar x)$ holds. Then the following statements are equivalent:

{\rm (i)} $\bar x$ is a strict local minimizer of order two for $f$;

{\rm (ii)} there exists $\beta>0$ such that
\begin{equation}\label{4.1}
f_-''(\bar x, 0, h)\geq \beta\,\,\,\,\forall h\in \mathbb{R}^m\,\,{\it with}\,\,\|h\|=1;
\end{equation}

{\rm (iii)} $
f_-''(\bar x, 0, h)>0$ holds for all $h\in \mathbb{R}^m\,\,{\it with}\,\,\|h\|=1.
$
}\\

Clearly it is shown from Theorem C that the second-order optimality condition of the first kind is necessary and sufficient for strict local minimizer of order two in the finite-dimension space setting. Further, Eberhard and Wenczel \cite{11} provide some conditions under which the second-order optimality conditions of the second and the third kinds are also necessary and sufficient for the existence of strict local minimizers of order two. Naturally, one question arisen here is whether or not the same results as in Theorem C are still valid for the case of Hilbert space. Unfortunately, the following example shows that the answer to this question is negative.\\

\noindent{\bf Example 5.1.} Let $H:=l^2$, and $e_k=(0,\cdots,1,0,\cdots)$ for each natural number $k$. We define a function $f: H\rightarrow \mathbb{R}\cup \{+\infty\}$ as:
$$
f(x)=\left \{
\begin{array}l
\,\,\frac{1}{k^3}\,\,, \,\,x=\frac{1}{k}e_k, k=1,2\cdots;\\
\,\,\,\,\,0\,\,\,,\, x=0;\\
+\infty, \ {\rm otherwise}.
\end{array}
\right.
$$
One can easily verify that $\bar x=0$ is a global minimizer of $f$ and $f_-''(0, 0, h)=+\infty$ for all $h\in S_H$. However, if we take $x_k:=\frac{1}{k}e_k$ for all
$k\in \mathbb{N}$, then $x_k\rightarrow 0$ and
$f(x_k)=\frac{1}{k}\cdot\frac{1}{k^2}$. This implies that
$$
\frac{f(x_k)-f(0)}{\|x_k-0\|^2}=\frac{1}{k}\rightarrow 0.
$$
Hence $\bar x=0$ is not the strict local minimizer of order two for $f$ even though $f$ satisfies the second-order optimality condition of the first kind at $\bar x$. \\

Next, we focus on characterizations for strict local minimizers of order two in a Hilbert space. To this aim, we consider the following notion.\\

\noindent{\bf Definition 5.2} {\it
Let $f:H\rightarrow \mathbb{R}\cup\{+\infty\}$ be a proper lower semicontinuous function, $\beta\in (0, +\infty)$, $p\in \partial_p f(\bar x)$ and let $A$ be a nonempty set of $H$. We say that $f_-''(\bar x, p, h)\geq \beta$ holds uniformly with respect to $h\in A$ if for any $\varepsilon>0$ there exists $\delta>0$ such that
\begin{equation}\label{4.4a}
\frac{f(\bar x+th)-f(\bar x)-t\langle p, h\rangle}{\frac{1}{2}t^2}\geq \beta-\varepsilon
\end{equation}
holds for all $t\in (0, \delta)$ and $h\in A+\delta B_H$.}\\

\begin{pro}
Let $f:H\rightarrow \mathbb{R}\cup\{+\infty\}$ be a proper lower semicontinuous function, $p\in \partial_p f(\bar x)$  and let $A$ be a compact set of $H$. Then the following statements are equivalent:

{\rm (i)} there exists $\beta>0$ such that $f_-''(\bar x, p, h)\geq \beta$ holds uniformly with respect to $h\in A$;

{\rm (ii)} there exists $\beta>0$ such that
\begin{equation}\label{4.5a}
f_-''(\bar x, p, h)\geq \beta\ \ {\it for\ all} \ h\in A;
\end{equation}

{\rm (iii)} $f_-''(\bar x, p, h)>0$ holds for all $h\in A$.
\end{pro}

\begin{proof}
(i)$\Rightarrow$ (ii): The implication follows from \eqref{4.4a} and definition of $f_-''(\bar x, 0, h)$.

(ii)$\Rightarrow$ (i): Let $h\in A$ and $\varepsilon>0$. By \eqref{4.5a}, there exists $\delta_h>0$ such that
\begin{equation}\label{4.6a}
\frac{f(\bar x+th')-f(\bar x)-t\langle p, h'\rangle}{\frac{1}{2}t^2}\geq \beta-\varepsilon\ \ \forall t\in (0, \delta_h)\ {\rm and}\ \forall h'\in B(h, \delta_h).
\end{equation}
Noting that $A$ is compact, there exist $h_1,\cdots, h_n\in A$ such that
\begin{equation}\label{4.7a}
A\subset \bigcup_{i=1}^nB(h_i, \frac{\delta_{h_i}}{2}).
\end{equation}
Let $\delta:=\min\{\frac{\delta_{h_1}}{2},\cdots,\frac{\delta_{h_n}}{2}\}$ and take arbitrary $t\in (0, \delta)$, $h'\in A+\delta B_H$. By virtue of \eqref{4.7a}, there exists $j\in \{1, \cdots, n\}$ such that
$$
h'\in B(h_j, \frac{\delta_{h_j}}{2})+\delta B_H\subset B(h_j, \delta_{h_j}).
$$
This and \eqref{4.6a} imply that
$$
\frac{f(\bar x+th')-f(\bar x)-t\langle p, h'\rangle}{\frac{1}{2}t^2}\geq \beta-\varepsilon.
$$
Thus (ii) holds.

(iii)$\Rightarrow$ (ii): By Proposition 3.1, one has that $f_-''(\bar x, p, \cdot)$ is lower semicontinuous, and there exists $\hat h\in A$ such that
$$
\min_{h\in A}f_-''(\bar x, p, h)=f_-''(\bar x, p, \hat h)>0
$$
(thanks to the compactness of $A$). Then (ii) follows by choosing $\beta:=f_-''(\bar x, p, \hat h)>0$. Since  (ii) implies (iii) trivially, the proof is completed.
\end{proof}

The following theorem provides characterizations for strict local minimizer of order two in the Hilbert space setting.

\begin{them}
Let $f: H\rightarrow \mathbb{R}\cup \{+\infty\}$ be a proper lower semicontinuous function and $\bar x\in{\rm dom}f$. Then the following statements are equivalent:

{\rm (i)} $\bar x$ is a strict local minimizer of order two for $f$.

{\rm (ii)} $0\in \partial_pf(\bar x)$ and there exists $\beta>0$ such that $f_-''(\bar x, 0, h)\geq \beta$ holds uniformly with respect to $h\in S_H$.

{\rm (iii)} The following inequality holds:
\begin{equation*}
\liminf_{t\downarrow 0}\Big(\inf_{h\in S_H}\frac{2(f(\bar x+th)-f(\bar x))}{t^2}\Big)>0.
\end{equation*}
\end{them}

\begin{proof} (i)$\Rightarrow$(ii): Since $\bar x$ is a strict local minimizer of order two for $f$, there exist $\beta, \delta>0$ such that
\begin{equation}\label{4a}
f(x)\geq f(\bar x)+\frac{\beta}{2}\|x-\bar x\|^2\ \ \forall x\in B(\bar x, \delta).
\end{equation}
This implies that $0\in\partial_pf(\bar x)$. Let $\varepsilon>0$ and take $\delta_1\in (0, \delta)$ such that $\delta_1<1$, $\delta_1(1+\delta)<\delta$ and $\beta(1-\delta_1)^2>\beta-\varepsilon$. Then for any $t\in(0, \delta_1)$ and $h\in S_H+\delta_1B_H$, one has $1-\delta_1\leq\|h\|\leq 1+\delta_1$ and it follow from \eqref{4a} that
\begin{eqnarray*}
\frac{f(\bar x+th)-f(\bar x)}{\frac{1}{2}t^2}\geq\beta\|h\|^2\geq\beta(1-\delta_1)^2>\beta-\varepsilon.
\end{eqnarray*}
Thus $f_-''(\bar x, 0, h)\geq\beta$ holds uniformly with respect to $h\in S_H$ and (ii) holds.

(ii)$\Rightarrow$(i): Let $\varepsilon\in (0, \beta)$. By Definition 5.2, there exists $\delta>0$ such that
\begin{equation}\label{2ab}
  \frac{f(\bar x+th)-f(\bar x)}{\frac{1}{2}t^2}\geq\beta-\varepsilon\ \ \forall t\in(0, \delta)\ {\rm and}\ \forall h\in S_H+\delta B_H.
\end{equation}
Then for any $x\in B(\bar x, \delta)\backslash\{\bar x\}$,  by \eqref{2ab}, one has
\begin{eqnarray*}
f(x)-f(\bar x)=f\big(\bar x+\|x-\bar x\|\cdot\frac{x-\bar x}{\|x-\bar x\|}\big)-f(\bar x)\geq\frac{\beta-\varepsilon}{2}\|x-\bar x\|^2.
\end{eqnarray*}
This implies that
$$
f(x)\geq f(\bar x)+\frac{\beta-\varepsilon}{2}\|x-\bar x\|^2\ \ \forall x\in B(\bar x, \delta).
$$
Hence $\bar x$ is a strict local minimizer of order two for $f$.

(i)$\Rightarrow$(iii): Let $\beta, \delta>0$ be such that \eqref{4a} hold. Then
$$
\inf_{h\in S_H}\frac{2(f(\bar x+th)-f(\bar x))}{t^2}\geq\beta \ \ \forall t\in (0, \delta)
$$
and consequently
$$
\liminf_{t\downarrow 0}\Big(\inf_{h\in S_H}\frac{2(f(\bar x+th)-f(\bar x))}{t^2}\Big)\geq\beta>0.
$$
This means that (iii) holds.

(iii)$\Rightarrow$(i):  Define
$$
\beta:=\liminf_{t\downarrow 0}\Big(\inf_{h\in S_H}\frac{2(f(\bar x+th)-f(\bar x))}{t^2}\Big)>0
$$
and let $\varepsilon\in (0, \beta)$. Then there exists $\delta>0$ such that
\begin{equation}\label{3a}
  \inf_{h\in S_H}\frac{2(f(\bar x+th)-f(\bar x))}{t^2}\geq\beta-\varepsilon\ \ \forall t\in (0, \delta).
\end{equation}
Let $x\in B(\bar x, \delta)\backslash\{\bar x\}$. By \eqref{3a}, one has
\begin{eqnarray*}
f(x)-f(\bar x)=f\big(\bar x+\|x-\bar x\|\cdot\frac{x-\bar x}{\|x-\bar x\|}\big)-f(\bar x)\geq\frac{\beta-\varepsilon}{2}\|x-\bar x\|^2.
\end{eqnarray*}
Hence $\bar x$ is a strict local minimizer of order two for $f$. The proof is complete.
\end{proof}

\noindent{\bf Remark 5.1.} We are now back to the Example 5.1. It is shown that $\bar x=0\in l^2$ is not the strict local minimizer of order two for $f$ appearing in Example 5.1. Further, if we take $t_k=\frac{1}{k}$ for each $k$, by computing, one has
$$
\inf_{h\in S_H}\frac{2(f(\bar x+t_kh)-f(\bar x))}{t_k^2}=\frac{2f(t_ke_k)}{t_k^2}=\frac{2}{k}.
$$
This means that
$$
\liminf_{t\downarrow 0}\Big(\inf_{h\in S_H}\frac{2(f(\bar x+th)-f(\bar x))}{t^2}\Big)\leq \lim_{k\rightarrow \infty}\Big(\inf_{h\in S_H}\frac{2(f(\bar x+t_kh)-f(\bar x))}{t_k^2}\Big)=0.
$$
Using Theorem 5.1, it follows that $\bar x=0\in l^2$ is not the strict local minimizer of order two for $f$.\\

Since the unit sphere of finite-dimensional space is compact, the following corollary is immediate from Proposition 5.1 and Theorem 5.1. This result also shows that Theorem C can be obtained from Theorem 5.1.

\begin{coro}
Let $H$ be a finite-dimensional space, $f:H\rightarrow \mathbb{R}\cup\{+\infty\}$ be a proper lower semicontinuous function and let $\bar x\in {\rm dom}f$. Then the following statements are equivalent:

{\rm (i)} $\bar x$ is a strict local minimizer of order two for $f$;

{\rm (ii)} $0\in \partial_pf(\bar x)$ and there exists $\beta>0$ such that $f_-''(\bar x, 0, h)\geq \beta$ holds uniformly with respect to $h\in S_H$;

{\rm (iii)} $0\in \partial_pf(\bar x)$ and there exists $\beta>0$ such that
$f_-''(\bar x, 0, h)\geq \beta$ holds for all $h\in S_H$;

{\rm (iv)} $0\in \partial_pf(\bar x)$ and $f_-''(\bar x, 0, h)>0$ holds for all $h\in S_H$.
\end{coro}

The following corollary, immediate from Theorem 5.1 and (i)$\Rightarrow$(ii) in Proposition 5.1, shows that the second-order optimality condition of the first kind is necessary for strict local minimizers of order two in the Hilbert space setting.

\begin{coro}
Let $f: H\rightarrow \mathbb{R}\cup\{+\infty\}$ be a lower semicontinuous function
and $\bar x\in{\rm dom}f$. If $\bar x$ is a strict local
minimizer of order two for $f$, then $0\in
\partial_pf(\bar x)$ and $f$ satisfies the second-order optimality condition of the first kind at $\bar x$.
\end{coro}

The following theorem, as one main result of this paper, is obtained from Corollary 5.2 and Theorems 4.1-4.3.

\begin{them}
Let $f:H\rightarrow \mathbb{R}\cup\{+\infty\}$ be a proper lower
semicontinuous and paraconcave function and $\bar x\in{\rm dom}f$ with $0\in \partial_pf(\bar x)$. Suppose that $f$ is continuous at $\bar x$ and twice epi-differentiable at $\bar x$ for $0$, and ${\rm qri}(N_M({\rm gph}(\partial_p f), (\bar x, 0)))\not=\emptyset$. If $\bar x$ is a strict
local minimizer of order two of $f$, then  $f$ satisfies all three types of second-order optimality conditions at $\bar x$.
\end{them}

\section{Conclusions}
This paper is devoted to second-order optimality conditions as well as applications in the Hilbert space. Three types of second-order derivatives of nonsmooth functions are considered to discuss these second-order optimality conditions. Their equivalence for paraconcave functions are also proved. As applications, these optimality conditions are used to study the strict local minimizer of order two of nonsmooth functions and provide its necessary and/or sufficient conditions. The work in this paper generalizes and extends the study of second-order optimality conditions from finite-dimensional space to the Hilbert space setting.\\

\noindent{\bf Acknowledgement.} The authors are grateful to Professor Xi Yin Zheng for his helpful suggestions on Definition 5.2 and Theorem 5.1.

\end{document}